\documentclass[10pt]{article}
\usepackage{mathrsfs}
\usepackage{amsfonts}
\usepackage{amsthm,amsmath,amssymb,anysize}
\newtheorem{lemma}{Lemma}[section]
\newtheorem{theorem}[lemma]{Theorem}

\newtheorem{proposition}[lemma]{Proposition}
\newtheorem{definition}[lemma]{Definition}

\usepackage[numbers,sort&compress]{natbib}
\marginsize{3.6cm}{3.6cm}{3.6cm}{3cm}
\numberwithin{equation}{section}

\title{\textsf{Multiplicative Hom-Lie superalgebra structures on infinite dimensional simple
Lie superalgebras of vector fields}}
\author{\textsc{Jixia Yuan$^{1,}$}\footnote{Supported by the NSF
  of  HLJ Provincial Education Department, China (12521158)}\;\;
\textsc{Liping Sun$^{2,}$}\footnote{Supported by the NSF  of HLJ
Provincial Education Department (12511349), China}\textsc{ and Wende
Liu$^{3,}$}\footnote{Corresponding author. Email:
\texttt{wendeliu@ustc.edu.cn}.\;\;Supported by the NSF
  of China (11171055)}\;  \\
  \\
  \textit{$^{1}$School of Mathematical Sciences},
  \textit{Heilongjiang University} \\
  \textit{Harbin 150080, China}\\
\\
  \ \ \textit{$^{2}$School of Applied Sciences},
  \textit{Harbin University of Science and Technology}\\
  \textit{Harbin 150080, China}\\
  \\
    \ \ \textit{$^{3}$School of Mathematical Sciences},
  \textit{Harbin Normal University} \\
  \textit{Harbin 150025, China}
  }
\date{ }
\begin{document}
\maketitle
\begin{quotation}
\noindent\textbf{Abstract} This paper considers  the multiplicative Hom-Lie
superalgebra structures on infinite dimensional simple
Lie superalgebras of vector fields with characteristic zero. The main
result is that there is only the  multiplicative Hom-Lie superalgebra structure on these
Lie superalgebras.\\



  \end{quotation}

  \setcounter{section}{-1}
\section{Introduction}

\noindent Hom-Lie algebra structures were introduced and studied in \cite{hls,ls1,ls2,ls3}. In 2008, Q. Jin and X. Li gave a description of Hom-Lie algebra structures of  Lie algebras and determined the isomorphic classes of nontrivial Hom-Lie algebra structures of  finite dimensional semi-simple  Lie algebras  \cite{jl}. The Hom-Lie algebras  have been sufficiently studied
in \cite{bm,s}.

 The theory of Lie superalgebras has  seen a significant
development. For example, V. G. Kac classified the finite
dimensional simple Lie superalgebras and the infinite dimensional
simple linearly compact Lie superalgebras over algebraically closed
fields of characteristic zero \cite{k1,k2}. In 2010, F. Ammar
and A. Makhlouf  generalized Hom-Lie algebras to Hom-Lie
superalgebras \cite{am}. In 2012, B. T. Cao and L. Luo proved
that there is only the  trivial Hom-Lie superalgebra structure on a
finite dimensional  simple  Lie  superalgebra of characteristic zero  \cite{cl}.

 This paper is motivated by the results and methods relative to   finite dimensional  simple  Lie  superalgebras with characteristic zero (cf. \cite{cl}).
 In Section 1 the notations of infinite dimensional simple
Lie superalgebras of vector fields were introduced. In Section 2 the multiplicative Hom-Lie superalgebra structures on  infinite dimensional simple
Lie superalgebras of vector fields were studied.
We proved that there is only the trivial multiplicative Hom-Lie superalgebra structure
 on infinite dimensional simple  Lie superalgebras of vector fields.

\section{Preliminaries}

Throughout $\mathbb{F}$ is a field of characteristic zero.
$\mathbb{Z}_2:= \{\bar{0},\bar{1}\}$ is the additive group of two
elements. $\mathbb{N}$ and $\mathbb{N}_0$ are the sets of positive
integers and nonnegative integers, respectively. $\mathbb{F}[x_{1},
\ldots, x_{m}]$ denotes the  polynomial algebra over $\mathbb{F}$ in
even indeterminates $x_{1}, x_{2},\ldots, x_{m},$ where $m>3$. For
positive integers $n>3,$ let $\Lambda(n)$ be the Grassmann
superalgebra over $\mathbb{F}$ in the $n$ odd indeterminates
$x_{m+1},x_{m+2},\ldots,x_{m+n}$.
  Clearly,
$$\Lambda(m,n):=\mathbb{F}[x_{1}, \ldots, x_{m}]\otimes
\Lambda(n).$$
is an associative commutative superalgebra.

 Let
$\partial_r$ be the superderivation of
$\Lambda(m,n) $ defined by
$\partial_{r}(x_{s})=\delta_{rs}$ for $r,s\in
\overline{1, m+n}.$ The \textit{generalized Witt superalgebra} $
W\left(m,n\right)$ is  $\mathbb{F}$-spanned by
$\{f_r
\partial_r\mid f_r\in \Lambda(m,n), r\in
\overline{1,n+m}\}.$   Note that $W(m,n) $ is a free $
\Lambda \left(m,n\right)$-module with basis $ \{
\partial_r\mid r\in
 \overline{1,m+n}\}.$

 For a   vector superspace $V=V_{\bar{0}}\oplus
V_{\bar{1}}, $ we write $|x|:=\theta$ for the \textit{parity} of a
homogeneous element $x\in V_{\theta}, $ $\theta\in \mathbb{Z}_{2}.$
Once the symbol $|x|$ appears, it
 will imply that $x$ is  a $\mathbb{Z}_2$-homogeneous element.

 The following symbols will be frequently used in this paper:
\begin{itemize}
  \item  $i_{H}^\prime=i_{K}^\prime:=\left\{
\begin{array}{lll}
i+r, &if\quad  1\leq i\leq r\\
 i-r,  & if\quad r<i\leq 2r\\
 i, & if\quad i\in \overline{m+1,m+n},
\end{array}
\right. \;\;\tau(i):=\left\{
\begin{array}{lll}
1, &if\quad  1\leq i\leq r\\
-1, & if\quad r<i\leq 2r\\
1, &if\quad i\in \overline{m+1,m+n},
\end{array}
\right.$ for  $m=2r$ or $m=2r+1;$
\item $i'_{X}:= \left\{\begin{array}{ll} i+m, &\mbox{if}\;
i\in\overline{1, m}
\\i-m, &\mbox{if}\; i\in \overline{m+1, 2m},
\end{array}\right.$ where $X=HO, KO, SHO$ and $SKO;$
\item $
\mathrm {div} ( f_{r}\partial_{r} )=(-1) ^{|\partial_r|
|f_{r}|}
\partial_{r}(f_{r}),
$ where $\mathrm {div}$ is a linear mapping from $W(m,n)$ to
$\Lambda(m,n);$
 \item $
\mathrm{div}_{\lambda}(f):=
(-1)^{|f|}2\left(\sum\limits_{i=1}^{m}\partial_i\partial_{i'_{SKO}}\left(f\right)+\left(\mathfrak{D}-m\lambda\mathrm{id}_{\Lambda(m,m+1)}\right)
\partial_{2m+1}\left(f\right)\right),
$ where $f\in
 \Lambda(m,m)$ and $\lambda\in \mathbb{F};$
\item $\mathrm{D}_{ij}(f):=(-1)^{|\partial_i||\partial_j|}\partial_i(f)\partial_j-(-1)^{(|\partial_i|+|\partial_j|)|f|}
\partial_j(f)\partial_i,$ where $f\in
 \Lambda(m,m);$
\item $
\mathrm{D_H}(f):=\sum\limits_{i=1}^{m+n}\tau(i)(-1)^{|\partial_i||f|}\partial_i(f)\partial_{i_{H}^\prime},
$ where $m=2r$ and $f\in
 \Lambda(m,m);$
\item $
\mathrm{ D_K}(f):=\sum\limits_{m\neq i=1}^{m+n}(-1)^{|\partial_i||f|}\left(x_i\partial_m(f)+\tau(i_{K}^\prime)\partial_{i_{K}^\prime}(f)\right)
\partial_i\\
+\Big(2f-\sum\limits_{m\neq i=1}^{m+n}x_i\partial_i(f)\Big)\partial_m,
$ where $m=2r+1$ and $f\in
 \Lambda(m,m);$
\item $\mathrm{D_{HO}}(f):=\sum\limits_{i=1}^{2m}(-1)^{|\partial_i||f|}\partial_i(a)\partial_{i'_{HO}},$
where $f\in
 \Lambda(m,m);$
 \item $\mathrm{D_{KO}}(f):=\mathrm{\mathrm{D_{HO}}}(f)+(-1)^{|a|}\partial_{2m+1}(f)\mathfrak{D}+ (\mathfrak{D}(f)-2f )\partial_{2m+1},\quad
\mathfrak{D}:=\sum\limits_{i=1}^{2m}x_{i}\partial_{i},$  where $f\in
 \Lambda(m,m);$
 \item $\nu:=\delta_{X,K}m+\delta_{X,KO}(2m+1)+\delta_{X,SKO}(2m+1).$
\end{itemize}

The following infinite dimensional Lie superalgebras of vector fields,
which are the simple Lie superalgebra contained in $W(m,n),$ are
defined as follows (cf. \cite{k2}):
\begin{itemize}
  \item $S\left(m,n \right):={\rm span}_{\mathbb{F}}\left\{\mathrm{D_{ij}}(f)\mid
f\in \Lambda\left(m,n \right), i,j\in   \overline{1,m+n}\right\};$
  \item $H\left(m,n\right):=\left\{\mathrm{D_H}(f)\mid f\in
\Lambda\left(m,n\right)\right\};$
  \item $K\left(m,n\right):=\left\{\mathrm{D_K}(f) \mid f\in \Lambda\left(m,n\right)\right\};$
  \item $HO\left(m,m\right):=\{\mathrm{D_{HO}}(f)\,|\,f\in\Lambda(m,m)\};$
  \item $
KO(m,m+1):=\{\mathrm{D_{KO}}(f)\mid f\in
\Lambda(m,m+1)\};
$
\item $
SHO(m,m):=[SHO'(m,m),SHO'(m,m)],
$ where
\\$
 SHO'(m,m):=\{D\in HO(m,m)\mid
\mathrm{div}(D)=0\};
$
\item $
SKO(m,m+1;\lambda):=
[SKO'(m,m+1;\lambda),
SKO'(m,m+1;\lambda)],$ where \\
$SKO'(m,m+1;\lambda):=\{\mathrm{D_{KO}}(f)\mid
\mathrm{div}_{\lambda}(f)=0, f\in \Lambda(m,m+1)\}$ and $\lambda\in \mathbb{F}.$
\end{itemize}
Hereafter,  write  $X$ for $W,$ $S,$ $H$, $K$, $HO$, $KO,$ $SHO$ or
$SKO.$
\section{Multiplicative Hom-Lie superalgebra}
\begin{definition}
A multiplicative Hom-Lie superalgebra is a triple $(\mathfrak{g},[,],\sigma)$ consisting of a $\mathbb{Z}_{2}$-graded vector space $\mathfrak{g}$, a bilinear map
$[,]:\mathfrak{g}\times\mathfrak{g}\longrightarrow \mathfrak{g}$ and an even linear map $\sigma:\mathfrak{g}\longrightarrow\mathfrak{g}$ satisfying
\begin{eqnarray}\label{e1}
&& \sigma[x,y]=[\sigma(x),\sigma(y)],\\\nonumber
&& [x,y]=-(-1)^{|x||y|}[y,x],\\
&&(-1)^{|x||z|}[\sigma(x),[y,z]]+(-1)^{|y||x|}[\sigma(y),[z,x]]+(-1)^{|z||y|}[\sigma(z),[x,y]]=0,\label{e3}
\end{eqnarray}
where $x, y$ and $z$ are homogeneous elements in $\mathfrak{g}.$
\end{definition}
For any simple Lie superalgebra $\mathfrak{g},$ denote its Lie bracket by
$[,]$  and take an even linear map
$\sigma:\mathfrak{g}\longrightarrow\mathfrak{g}.$ We say
$(\mathfrak{g},\sigma)$ is a multiplicative Hom-Lie superalgebra structure over the
Lie superalgebra $\mathfrak{g}$ if $(\mathfrak{g},[,],\sigma)$ is a multiplicative
Hom-Lie superalgebra. Suppose $\sigma\neq0 $.  Eq. (\ref{e1}) and the
simplicity of $\mathfrak{g}$ show that $\sigma$ is a monomorphism of
$\mathfrak{g}.$ In particularly, if $\sigma=\mathrm{id}$  or $\sigma=0$, the multiplicative
Hom-Lie superalgebra structure is called trivial.
Before consider the multiplicative Hom-Lie superalgebra structures on $X(m,n),$ we
introduce the gradations on them as in \cite{k2}. For any
$(m+n)$-tuple $\underline{\alpha}:=(\alpha_{1},\ldots,\alpha_{m}\mid \alpha_{m+1},\ldots,\alpha_{m+n})\in
\mathbb{N}^{m+n},$ we may define a gradation on $W(m,n)$ by letting $\mathrm{deg}x_{i}:=\alpha_{i}=:-\mathrm{deg}\partial_{i},$ where $i\in \overline{1,m+n}.$ Thus $W(m,n)$   becomes a graded Lie superalgebra of finite depth, i.e., we have $$W(m,n)=\oplus_{j=-h}^{\infty}W(m,n)_{\underline{\alpha},[j]},$$
 where $h$ is a positive integer.
 Put
$$\underline{\gamma}:=\underline{1}+\delta_{X,K}\varepsilon_{m}+\delta_{X,KO}\varepsilon_{2m+1}+\delta_{X,SKO}\varepsilon_{2m+1}\in \mathbb{N}^{m+n}$$
and
sometimes  omit the subscript
$\underline{\gamma}$. Putting
$$X(m,n)_{\underline{\gamma},[i]}:=X(m,n)\cap W(m,n)_{\underline{\gamma},[i]},$$ one sees that $X(m,n)$ is graded by
$(X(m,n)_{\underline{\gamma},[i]})_{i\in \mathbb{Z}}.$
In particular,
\begin{itemize}
  \item $X(m,n)_{[-2]}=\mathbb{F}\cdot \mathrm{D_{X}}(1),$ where $X:=K, KO$ or $SKO;$
  \item $X(m,n)_{[-1]}=\mathrm{span}_{\mathbb{F}}\{\partial_{i}\mid
i\in \overline{1,m+n}\},$ where $X:=W$ or $S;$
  \item $X(m,n)_{[-1]}=\mathrm{span}_{\mathbb{F}}\{\mathrm{D_{X}}(x_{i})\mid
\nu\neq i\in \overline{1,2n}\},$ where $X:=H, K, HO, KO, SHO$ or $SKO;$
  \item $W(m,n)_{[0]}=\mathrm{span}_{\mathbb{F}}\{x_{i}\partial_{j}\mid i,j\in \overline{1,m+n}\};$
  \item $S(m,n)_{[0]}=\mathrm{span}_{\mathbb{F}}\{x_{i}\partial_{j}, x_{i}\partial_{i}-x_{j}\partial_{j}\mid i\neq j\in \overline{1,m+n}\};$
  \item $X(m,n)_{[0]}=\mathrm{span}_{\mathbb{F}}\{\mathrm{D_{H}}(x_{i}x_{j}), \delta_{X, K}\mathrm{D_{H}}(x_{m}), \delta_{X, KO}\mathrm{D_{H}}(x_{2m+1})\mid i, j\in \overline{1,m+n}\},$ where $X:=H, K, HO$ or $KO;$
  \item $X(m,n)_{[0]}=\mathrm{span}_{\mathbb{F}}\{
\mathrm{D_{X}}(x_{i}x_{j}), \mathrm{D_{X}}(x_{i}x_{i'}-x_{j}x_{j'}),
\mathrm{D_{X}}(x_{2n+1}+\delta_{X,SKO}n\lambda x_{i}x_{i'})\mid i\neq j'\in
\overline{1,m+n}\},$ where $X:=SHO$ or $SKO.$
\end{itemize}

Next we  give an equation and several lemmas needed in the sequel.
 The verifications  are
straightforward. The equation will be used without notice: for $f,
g\in \Lambda(m,n)$,
\begin{eqnarray*}
\left[\mathrm{D_{X}}(f),\mathrm{D_{X}}(g)\right]=\mathrm{D_{X}}\left(\mathrm{D_{X}}\left(f\right)\left(g\right)
-2\left(\delta_{X,K}-(-1)^{|f|}\delta_{X,KO}\right)\partial_{\nu}\left(f\right)g\right).\label{liuee3}
\end{eqnarray*}
\begin{lemma}(cf. \cite{k2})\label{l3}
The  $\mathbb{Z}$-graded Lie superalgebra $X(m,n)$ is transitive,
that is, if $x \in \mathfrak{g}_{i}$ with $i\geq0$ and
$[x,\mathfrak{g}_{[-1]}]=0$, then $x=0$.
\end{lemma}
\begin{lemma}\label{ll3}
 For $\mathbb{Z}$-graded Lie superalgebra $X(m,n)$, we have
 $$\mathrm{Ker}({\mathrm{ad}\partial_{i}})\cap X(m,n)_{[0]}=\mathrm{span}_{\mathbb{F}}\{x_{j}\partial_{k}\mid j, k\in \overline{1,m+n}, i\neq j\}\cap X(m,n)_{[0]}$$ and
 $$[\mathrm{Ker}({\mathrm{ad}\partial_{i}})\cap X(m,n)_{[0]}, \mathrm{Ker}({\mathrm{ad}\partial_{i}})\cap X(m,n)_{[0]}]=\mathrm{Ker}({\mathrm{ad}\partial_{i}})\cap X(m,n)_{[0]},$$
 where $i\in \overline{1,m+n}\backslash \nu $.
\end{lemma}

\begin{lemma}\label{yuanl1}
For $i, j, k, l\in \overline{1,m+n}$, we have that
$$x_{k}\partial_{l}\in [\mathrm{Ker}(\mathrm{ad}x_{i}\partial_{j})\cap X(m,n)_{[0]},\mathrm{Ker}(\mathrm{ad}x_{i}\partial_{j})\cap X(m,n)_{[0]}],$$
 where $k\neq j, l;$
$$\mathrm{D_{X}}(x_{k}x_{l})\in [\mathrm{Ker}(\mathrm{ad}\mathrm{D_{X}}(x_{i}x_{j}))\cap X(m,n)_{[0]},\mathrm{Ker}(\mathrm{ad}\mathrm{D_{X}}(x_{i}x_{j}))\cap X(m,n)_{[0]}],$$
where $k\neq l\in\overline{1,m+n}\backslash \nu,$ $k, l\neq i'_{X}, j'_{X}$.
\end{lemma}
 The next  proposition  is essential for the main  result in this paper.
 \begin{proposition}\label{p1}
If $(X(m,n),\sigma)$ is a multiplicative Hom-Lie superalgebra structure and $\sigma\neq 0$, then
$$\sigma\mid_{X(m,n)_{[-1]}}= \mathrm{id}\mid_{X(m,n)_{[-1]}}.$$
\end{proposition}
\begin{proof}
\textbf{Case 1}: $X:=W$ or $S$.  By Eq. (\ref{e3}), we have
\begin{eqnarray*}
0&=&(-1)^{|\partial_{i}||x_{j}\partial_{k}|}[\sigma(\partial_{i}),[\partial_{j},x_{j}\partial_{k}]]
+
(-1)^{|\partial_{i}||\partial_{j}|}[\sigma(\partial_{j}),[x_{j}\partial_{k},\partial_{i}]]\\
&&+
(-1)^{|\partial_{j}||x_{j}\partial_{k}|}[\sigma(x_{j}\partial_{k}),[\partial_{i},\partial_{j}]]\\
&=&(-1)^{|\partial_{i}||x_{j}\partial_{k}|}[\sigma(\partial_{i}),\partial_{k}],
\end{eqnarray*}
where $j\neq i, k\in \overline{1,m+n}.$ By Lemma \ref{l3}, we have
$$\sigma(X(m,n)_{[-1]})=X(m,n)_{[-1]}.$$
Then for any $i\in \overline{1,m+n},$ one may suppose $\sigma(\partial_{i})=\sum\limits_{l=1}^{m+n}a_{l}\partial_{l},$ where $a_{l}\in \mathbb{F}.$  Lemma \ref{ll3} and Eq. (\ref{e3}) imply that
$\sigma(\partial_{i})=a_{i}\partial_{i}.$ For distinct $i, j, k\in \overline{1,m+n},$ put $x=x_{j}\partial_{j}-x_{k}\partial_{k},$ $y=\partial_{i}$ and $z=x_{i}\partial_{j}.$  Then Eq. (\ref{e3}) implies that
\begin{eqnarray}\label{e88}
[\sigma(x_{j}\partial_{j}-x_{k}\partial_{k}), \partial_{j}]+[\sigma(\partial_{i}),x_{i}\partial_{j}]=0.
\end{eqnarray}
Suppose $\sigma^{-1}$ is  an left linear inverse of $\sigma$ (vector space). Then
\begin{eqnarray*}
\sigma^{-1}([\sigma(x_{j}\partial_{j}-x_{k}\partial_{k}),\partial_{j}])=[x_{j}\partial_{j}-x_{k}\partial_{k},\sigma^{-1}(\partial_{j})]=
[x_{j}\partial_{j}-x_{k}\partial_{k},a_{j}^{-1}\partial_{j}]=-a_{j}^{-1}\partial_{j}.
\end{eqnarray*}
Hence
$$[\sigma(x_{j}\partial_{j}-x_{k}\partial_{k}),\partial_{j}]=-\partial_{j}.$$
By Eq. (\ref{e88}), we have $a_{i}=1,$ where $i\in \overline{1,m+n}.$ That is
\begin{eqnarray*}
(\sigma-\mathrm{id})\mid_{X(m,n)_{[-1]}}=0.
\end{eqnarray*}

 \noindent\textbf{Case 2}: $X:=H, K,  HO,  KO,  SHO$ or $SKO.$ For $i,
j, k\in \overline{1,m+n}\backslash \nu,$ Eq. (\ref{e3}) implies that
\begin{eqnarray*}
0&=&(-1)^{|\mathrm{D_{X}}(x_{i})||\mathrm{D_{X}}(x_{j'_{X}}x_{k'_{X}})|}[\sigma(\mathrm{D_{X}}(x_{i})),[\mathrm{D_{X}}(x_{j}), \mathrm{D_{X}}(x_{j'_{X}}x_{k'_{X}})]]\\
&&+
(-1)^{|\mathrm{D_{X}}(x_{j})||\mathrm{D_{X}}(x_{i})|}[\sigma(\mathrm{D_{X}}(x_{j})),[\mathrm{D_{X}}(x_{j'_{X}}x_{k'_{X}}),\mathrm{D_{X}}(x_{i})]]\\
&&+
(-1)^{|\mathrm{D_{X}}(x_{j'_{X}}x_{k'_{X}})||\mathrm{D_{X}}(x_{j})|}[\sigma(\mathrm{D_{X}}(x_{j'_{X}}x_{k'_{X}})),[\mathrm{D_{X}}(x_{i}),\mathrm{D_{X}}(x_{j})]].
\end{eqnarray*}
It follows that
\begin{eqnarray}\label{e15}
[\sigma(\mathrm{D_{X}}(x_{i})),\mathrm{D_{X}}(x_{k'_{X}})]=0,\;\;i\neq j, j'_{X}, k
\end{eqnarray}
 and
 \begin{eqnarray}\label{e16}
 [\sigma(\mathrm{D_{X}}(x_{i})),\mathrm{D_{X}}(x_{i'_{X}})]=[\sigma(\mathrm{D_{X}}(x_{j})),\mathrm{D_{X}}(x_{j'_{X}})], \;\; i\neq j,j'.
 \end{eqnarray}
 By Eq. (\ref{e15}), (\ref{e16}) and Lemma \ref{l3}, it is easy to obtain that
 $$\sigma(\mathrm{D_{X}}(x_{i}))=a_{i}\mathrm{D_{X}}(1)+ \sum_{\nu\neq l=1}^{m+n}a_{il}\mathrm{D_{X}}(x_{l})$$
for some $a_{i}, a_{il}\in \mathbb{F}.$\\

\noindent\textbf{Subcase 2.1}: $X:=H, HO,$ or $SHO.$  Lemma \ref{ll3} and Eq.
(\ref{e3}) imply that
 $a_{ik}=0$ for all $i\neq k\in \overline{1,m+n}\backslash \nu.$ Thus,
\begin{eqnarray*}
(\sigma-\mathrm{id})\mid_{X(m,n)_{[-1]}}=0.
\end{eqnarray*}

\noindent\textbf{Subcase 2.2}: $X:=K, KO$ or $SKO.$ Put $x=\mathrm{D_{X}}(x_{i}),$
$y=\mathrm{D_{X}}(x_{j})$ and
$$z=\mathrm{D_{X}}(x_{j'_{X}}x_{\nu}+\delta_{X,SKO}(-1)^{|x_{j'_{X}}|}(m\lambda-1)x_{k}x_{k'_{X}}x_{j'_{X}}).$$
Eq. (\ref{e3}) implies that
\begin{eqnarray*}\label{e888}
0&=&(-1)^{|\mathrm{D_{X}}(x_{i}))||z|}[\sigma(\mathrm{D_{X}}(x_{i})), \mathrm{D_{X}}(x_{\nu}+\delta_{X,SKO}(m\lambda-1)x_{k}x_{k'_{X}}+x_{j}x_{j'_{X}})]\\
&&+(-1)^{|\mathrm{D_{X}}(x_{i})||\mathrm{D_{X}}(x_{j})|}[\sigma(\mathrm{D_{X}}(x_{j})),\mathrm{D_{X}}(x_{i}x_{j'_{X}})].
\end{eqnarray*}
Hence $a_{i}=0.$ Take  $i, j, k\in \overline{1,m+n}\backslash \nu$ and $i\neq j, j'_{X}, k, k'_{X}.$ Put  $x=\mathrm{D_{X}}(x_{j}x_{k})\in X(m,n)_{\bar{0}},$ $y=\mathrm{D_{X}}(x_{i})$ and $z=\mathrm{D_{X}}(x_{i'_{X}}x_{j'_{X}}).$  By  Eq. (\ref{e3})  again, we have
\begin{eqnarray*}\label{e888}
0&=&(-1)^{|\partial_{i}||x_{i}|}[\sigma(\mathrm{D_{X}}(x_{j}x_{k})), \mathrm{D_{X}}(x_{j'_{X}})]\\
&&+(-1)^{|\partial_{j'_{X}}||x_{i'_{X}}x_{j'_{X}}|+|\partial_{j'_{X}}||x_{i'_{X}}|}
[\sigma(\mathrm{D_{X}}(x_{i})),\mathrm{D_{X}}(x_{i'_{X}}x_{k})].
\end{eqnarray*}
Suppose $\sigma^{-1}$ is  a left inverse of $\sigma.$ Then
\begin{eqnarray*}
\sigma^{-1}([\sigma(\mathrm{D_{X}}(x_{j}x_{k})), \mathrm{D_{X}}(x_{j'_{X}})])&=&[\mathrm{D_{X}}(x_{j}x_{k}), \sigma^{-1}(\mathrm{D_{X}}(x_{j'_{X}}))]\\
&=&
[\mathrm{D_{X}}(x_{j}x_{k}), a_{j'_{X}j'_{X}}^{-1}\mathrm{D_{X}}(x_{j'_{X}})]\\
&=&-(-1)^{|\partial_{j'_{X}}||x_{j'_{X}}|}a_{j'_{X}j'_{X}}^{-1}\mathrm{D_{X}}(x_{k}).
\end{eqnarray*}
Hence
\begin{eqnarray*}
-(-1)^{|\partial_{j'_{X}}||x_{j'_{X}}|}a_{kk}a_{j'_{X}j'_{X}}^{-1}\mathrm{D_{X}}(x_{k})&=&[\sigma(\mathrm{D_{X}}(x_{j}x_{k})), \mathrm{D_{X}}(x_{j'_{X}})]\\
&=&-(-1)^{|\partial_{j'_{X}}||x_{j'_{X}}|+|\partial_{i}||x_{i}|}[\sigma(\mathrm{D_{X}}(x_{i})),\mathrm{D_{X}}(x_{i'_{X}}x_{k})]\\
&=&-(-1)^{|\partial_{j'_{X}}||x_{j'_{X}}|}a_{ii}\mathrm{D_{X}}(x_{k}).
\end{eqnarray*}
The arbitrariness of $j$ and $k$ implies that $a_{ii}=1.$
Thus,
\begin{eqnarray*}
(\sigma-\mathrm{id})\mid_{X(m,n)_{[-1]}}=0.
\end{eqnarray*}
\end{proof}
\begin{proposition}\label{p2}
If $(X(m,n),\sigma)$ is a multiplicative Hom-Lie superalgebra structure and $\sigma\neq 0$, then
$$\sigma\mid_{X(m,n)_{[0]}}= \mathrm{id}\mid_{X(m,n)_{[0]}}.$$
\end{proposition}
\begin{proof}
 \textbf{Case 1}: $X:=W$ or $S.$ Put  $x\in X(m,n)_{[0]}.$ Then by
Proposition \ref{p1}, we have
\begin{eqnarray*}\label{e99}
[\sigma(x),\partial_{i}]=[\sigma(x),\sigma(\partial_{i})]=\sigma([x,\partial_{i}])=[x,\partial_{i}]
\end{eqnarray*}
for all $i\in \overline{1,m+n}.$ By  Lemma \ref{l3},
 we may write
$$\sigma(x_{i}\partial_{j})=x_{i}\partial_{j}+\sum_{s=1}^{m+n}a_{ijs}\partial_{s},$$
where   $a_{ijs}\in \mathbb{F}.$
By Lemma \ref{yuanl1} and Eq. (\ref{e3}), we have
$$a_{ijk}\partial_{l}=[\sum_{s=1}^{m+n}a_{ijs}\partial_{s}, x_{k}\partial_{l}]=0$$
for  $k, l\in \overline{1,m+n}$ and $k\neq j,l.$
  By the arbitrariness of $l$, we know  $a_{ijk}=0$ for all $j\neq k\in \overline{1,m+n}.$  Put $x=x_{i}\partial_{j},$ $y=x_{j}\partial_{l},$ $z=x_{l}\partial_{l}-(-1)^{(|x_{l}|+|x_{s}|)}x_{s}\partial_{s}.$ By Eq. (\ref{e3}) we have that
$$[\sigma(x_{i}\partial_{j}),x_{j}\partial_{l}]+[\sigma(x_{l}\partial_{l}-(-1)^{(|x_{l}|+|x_{s}|)}x_{s}\partial_{s}),[x_{i}\partial_{j},x_{j}\partial_{l}]]=0.$$
Furthermore,
$$[a_{ijj}\partial_{j},x_{j}\partial_{l}]+[a_{lll}\partial_{l}-a_{sss}\partial_{s},x_{i}\partial_{l}]=0.$$
Then $a_{ijj}=0.$ Summarizing,  we have $\sigma(x_{i}\partial_{j})=x_{i}\partial_{j}.$
\\

\noindent\textbf{Case 2}: $X:=H,  K,  HO,   KO,   SHO$ or $SKO.$ For
$x\in X(m,n)_{[0]},$      by Proposition \ref{p1}, we have
\begin{eqnarray}\label{e22}
[\sigma(x),\mathrm{D_{X}}(x_{i})]=[\sigma(x),\sigma(\mathrm{D_{X}}(x_{i}))]=\sigma([x,\mathrm{D_{X}}(x_{i})])=[x,\mathrm{D_{X}}(x_{i})]
\end{eqnarray}
where  $i\in \overline{1,m+n}\backslash \nu.$
By Lemma \ref{l3}, we may write
$$\sigma(\mathrm{D_{X}}(x_{i}x_{j}))=\mathrm{D_{X}}(x_{i}x_{j})+
\sum_{\nu\neq s=1}^{m+n}a_{ijs}\mathrm{D_{X}}(x_{s})+a_{ij}\mathrm{D_{X}}(1),$$
where $\mathrm{D_{X}}(x_{i}x_{j})\in X(m,n)_{[0]}$ and $a_{ij}, a_{ijs}\in \mathbb{F}.$
 By  Lemma \ref{yuanl1} and Eq. (\ref{e3}) we have that
$$\pm a_{ijk'_{X}}\mathrm{D_{X}}(x_{l})\pm a_{ijl'_{X}}\mathrm{D_{X}}(x_{k})=[\sum_{s=1}^{m+n}a_{ijs}\mathrm{D_{X}}(x_{s}), \mathrm{D_{X}}(x_{k}x_{l})]=0,$$
where $k\neq l\in \overline{1,m+n}\backslash \nu$ and $k, l\neq i'_{X}, j'_{X}.$
That is $a_{ijs}=0$ for all $i, j\neq s\in \overline{1,m+n}\backslash \nu.$ Then
$$\sigma(\mathrm{D_{X}}(x_{i}x_{j}))=\mathrm{D_{X}}(x_{i}x_{j})
+a_{iji}\mathrm{D_{X}}(x_{i})+a_{ijj}\mathrm{D_{X}}(x_{j})+a_{ij}\mathrm{D_{X}}(1).$$
If $k=i'_{X}$ or $k=j'_{X},$ by Eq. (\ref{e3}) we have
$$
\pm a_{ijk}\mathrm{D_{X}}(x_{l})=[a_{iji}\mathrm{D_{X}}(x_{i})+a_{ijj}\mathrm{D_{X}}(x_{j}),\mathrm{D_{X}}(x_{k}x_{l})]=0
$$
Hence
\begin{eqnarray}\label{e20}
\sigma(\mathrm{D_{X}}(x_{i}x_{j}))=\mathrm{D_{X}}(x_{i}x_{j})+a_{ij}\mathrm{D_{X}}(1).
\end{eqnarray}

\noindent\textbf{Subcase 2.1}: $X:=H$ or $HO.$ From Eq. (\ref{e20}), we have
$$(\sigma-\mathrm{id})\mid_{X(m,n)_{[0]}}=0.$$

\noindent\textbf{Subcase 2.2}: $X:=SHO.$ By Eq. (\ref{e20}) again, for $i\neq
j\in \overline{1,m}\backslash \nu$ we can obtain
\begin{eqnarray}\label{e21}
\sigma(\mathrm{D_{X}}(x_{i}x_{i'_{X}}-x_{j}x_{j'_{X}}))&=&\sigma([\mathrm{D_{X}}(x_{i}x_{j}), \mathrm{D_{X}}(x_{i'_{X}}x_{j'_{X}})])\\\nonumber
&=&
[\sigma(\mathrm{D_{X}}(x_{i}x_{j})), \sigma(\mathrm{D_{X}}(x_{i'_{X}}x_{j'_{X}}))]\\\nonumber
&=&[\mathrm{D_{X}}(x_{i}x_{j}), \mathrm{D_{X}}(x_{i'_{X}}x_{j'_{X}})]=\mathrm{D_{X}}(x_{i}x_{i'_{X}}-x_{j}x_{j'_{X}}).
\end{eqnarray}
From Eq. (\ref{e20}) and (\ref{e21}), we know
$$(\sigma-\mathrm{id})\mid_{SHO(m,n)_{[0]}}=0.$$

\noindent\textbf{Subcase 2.3}: $X:=K, KO$ or $SKO.$ Take
$x=\mathrm{D_{X}}(x_{i}x_{j}),$ $y=\mathrm{D_{X}}(x_{k})$ and
$z=\mathrm{D_{X}}(x_{k'_{X}}x_{\nu}$ $+(-1)^{|x_{k'_{X}}|}(n\lambda-1)x_{l}x_{l'_{X}}),$
where $i, j, k, k'_{X}, l, l'_{X}\in \overline{1,m+n}$ are distinct.
By Eq. (\ref{e3}), we have
\begin{eqnarray*}
-2a_{ij}=[\sigma(x),[y,z]]&=&(-1)^{|\partial_{k}||x_{k}|}[\mathrm{D_{X}}(x_{i}x_{j})+a_{ij}\mathrm{D_{X}}(1), \\
&&\mathrm{D_{X}}(x_{\nu}+(-1)^{|x_{k'_{X}}|}(n\lambda-1)x_{l}x_{l'_{X}}+
(-1)^{|x_{k'_{X}}|}x_{k}x_{k'_{X}})]=0.
\end{eqnarray*}
Hence
$$
\sigma(\mathrm{D_{X}}(x_{i}x_{j}))=\mathrm{D_{X}}(x_{i}x_{j}).
$$
By Eq. (\ref{e22}), we can write
\begin{eqnarray}\label{e23}
\sigma(\mathrm{D_{X}}(x_{\nu}+\delta_{X,SKO}n\lambda x_{j}x_{j'_{X}}))&=&\mathrm{D_{X}}(x_{\nu}+\delta_{X,SKO}n\lambda x_{j}x_{j'_{X}})\\\nonumber
&&+\sum_{\nu\neq s=1}^{m+n}a_{\nu js}\mathrm{D_{X}}(x_{s})+a_{\nu j}\mathrm{D_{X}}(1).
\end{eqnarray}
Using the method above, we can obtain easily
$$
\sigma(\mathrm{D_{X}}(x_{\nu}+\delta_{X,SKO}n\lambda x_{j}x_{j'_{X}}))=\mathrm{D_{X}}(x_{\nu}+\delta_{X,SKO}n\lambda x_{j}x_{j'_{X}}).
$$
By Eq. (\ref{e20}), (\ref{e21}) and (\ref{e23}), we have
$$(\sigma-\mathrm{id})\mid_{X(m,n)_{[0]}}=0.$$
The proof is
complete.
\end{proof}

\begin{theorem}
There is only the trivial multiplicative Hom-Lie superalgebra structure  on the infinite dimensional
simple Lie superalgebras of vector fields.
\end{theorem}
\begin{proof} Let $(X(m,n),\sigma)$ be a multiplicative Hom-Lie superalgebra structure and $\sigma \neq 0.$ By Propositions \ref{p1} and \ref{p2}, we have    $$\sigma\mid_{X(m,n)_{[-1]}\oplus X(m,n)_{[0]}}=\mathrm{id}\mid_{X(m,n)_{[-1]}\oplus X(m,n)_{[0]}}.$$
Now let $x\in X(m,n)_{[l]}$ and  $y, z\in X(m,n)_{[-1]}\oplus X(m,n)_{[0]},$  where $l\geq1.$ By Eq. (\ref{e3}), we have
\begin{eqnarray*}
[\sigma(x)-x,[y,z]]=0.
\end{eqnarray*}
 Then
$\sigma(x)-x=0.$ We get $\sigma=\mathrm{id}.$ The proof is complete.
\end{proof}

\end{document}